\documentclass{amsart}
\usepackage{amssymb}
\usepackage{amscd}
\usepackage[all,cmtip]{xy}
\usepackage{hyperref}
\usepackage{graphicx}

\newcommand{\comment}[1]{}

\newcommand{\C}{\mathbf{C}}

\newcommand{\Z}{\mathbf{Z}}

\theoremstyle{plain}
\newtheorem{theorem}{Theorem}[section]

\newtheorem{lemma}[theorem]{Lemma}

\theoremstyle{definition}
\newtheorem{definition}[theorem]{Definition}
\newtheorem{remark}[theorem]{Remark}

\makeindex

\begin{document}

\begin{center}
{\large\bf Michiel Kosters - Mathematisch Instituut, Universiteit Leiden \par}
{\large\bf \texttt{mkosters@math.leidenuniv.nl}, \today \par}
 \vspace{3em} {\LARGE\bf The subset sum problem for finite abelian groups \par} 
{\large\bf }
\end{center}

\section{Abstract}

Let $G$ be a finite abelian group. For $g \in G$ and $i$ an integer we define \\
$N(i,g)=\#\{S \subseteq G: \#S=i, \sum_{s \in S}s=g\}$, the number of subsets of $G$ of size $i$ which sum up to $g$. We will give a short proof, using character theory, of a formula for these $N(i,g)$ due to Li and Wan. We also give a formula for $N(i,g)^*= \#\{S \subseteq G \setminus\{0\}: \#S=i, \sum_{s \in S}s=g\}$, which
generalizes another result of Wan.  \\

Keywords: subset sum, finite abelian group.

Classification MSC2010: 11B75, 20K99.

\section{Main theorem}

In this article we fix a finite abelian group $G$ of size $n$ and we let $0 \leq i \leq n$.

\begin{definition}
 For $g \in G$ we define $N(i,g)=\#\{S \subseteq G: \#S=i, \sum_{s \in S}s=g\}$.
\end{definition}

\begin{definition}
 For $g \in G$ we define $e(g)=\mathrm{max}\{d: d|\mathrm{exp}(G), g \in dG\}$. 
\end{definition}

Remark that $e(g)=\mathrm{lcm}\{d: d|\mathrm{exp}(G), g \in
dG\}$.
Indeed, if $g=d_1 g_1=d_2 g_2$ with $g_1, g_2 \in G$ and $d_1,d_2 \in \Z$ with $\gcd(d_1,d_2)=1$, then there are integers such that $1=n_1d_1+n_2d_2$.
Hence we have
\[g=(n_1d_1+n_2d_2)g=d_1d_2(n_1g_1+n_2g_2).\]

Let $\mu$ be the M\"obius function and for an integer $d$ let $G[d]=\{h \in G: dh=0\}$, the $d$-torsion of $G$.

Our main theorem is the following one.

\begin{theorem} \label{la}
 We have the following formula for $g \in G$:
\begin{eqnarray*}
 N(i,g)=\frac{1}{n} \sum_{s| \gcd(\mathrm{exp}(G),i)} (-1)^{i+\frac{i}{s}} {{n/s}\choose{i/s}} \sum_{d|\gcd(e(g),s)} \mu(\frac{s}{d}) \#G[d].
\end{eqnarray*}
\end{theorem}

\begin{remark}
The above theorem is a slight improvement of Theorem 1.1 in \cite{WA}. In \cite{WA} a rather long proof is given using some sieving techniques. 
\end{remark}

Let $\hat{G}$ be the set of characters on $G$ with identity $\chi_0=1$. Note that $\hat{G}$ is a group isomorphic to $G$. We
make the following observation:
\begin{eqnarray*}
 \sum_{i=0}^n \sum_{g \in G} N(i,g)g X^i= \prod_{\sigma \in G} (1+\sigma X) \in \C[G][X].
\end{eqnarray*}
We use the following easy fact. If $\alpha=\sum_{g \in G} \alpha_g g \in \C[G]$, then $\alpha_g=\frac{1}{n} \sum_{\chi
\in
\hat{G}}
\overline{\chi}(g)\chi(\alpha)$ (here we extend $\chi$ to a $\C$-linear map $\chi: \C[G] \to \C$). 
This gives the following formula, where we put $Y=-X$:
\begin{eqnarray*} \label{r}
\sum_{i=0}^n N(i,g)(-Y)^i= \frac{1}{n} \sum_{\chi \in \hat{G}} \overline{\chi}(g) \prod_{\sigma \in G}\left( 1-\chi(\sigma)Y \right).  
\end{eqnarray*}

\begin{lemma}
 Suppose that $\chi \in \hat{G}$ has order $m$. Then $\prod_{\sigma \in G}\left( 1-\chi(\sigma)Y \right)=(1-Y^m)^{n/m}$. 
\end{lemma}
\begin{proof}
Observe that $\prod_{i=0}^m (1-\zeta_m^iY)=1-Y^m$, and that we have $n/m$ of such products. 
\end{proof}

\begin{lemma}
We have
\begin{eqnarray*}
 \sum_{\chi \in \hat{G}: \chi^m=1} \chi(g) = \left\{ \begin{array}{cc}
                                        \#G[m]  	& g \in mG \\
					0	& g \not \in mG.
                                        \end{array}\right.
\end{eqnarray*}
\end{lemma}
\begin{proof}
If $\chi^m= 1$, we know that the character factors through $G/mG$. A character on $G/mG$ induces such a character on $G$. Hence we deduce the
result from standard character theory. 
\end{proof}

Now define the following function:
\begin{eqnarray*}
 f(s)= \sum_{d|s} \sum_{\chi \in \hat{G}: \mathrm{ord}(\chi)=d} \overline{\chi}(g).
\end{eqnarray*}
From the lemma above we have $f(s)= \delta_{g \in sG} \#G[s]$.
Using the M\"obius inversion formula we find, for $s|\mathrm{exp}(G)$,
\begin{eqnarray*}
 \sum_{\chi \in \hat{G}: \mathrm{ord}(\chi)=s} \overline{\chi}(g) &=& \sum_{d|s} \mu(s/d) \delta_{g \in dG} \#G[d] \\
&=& \sum_{d|(s,e(g))} \mu(s/d) \#G[d].
\end{eqnarray*}

Hence we find
\begin{eqnarray} \label{s}
\sum_{i=0}^n N(i,g)(-Y)^i= \frac{1}{n} \sum_{s|\mathrm{exp}(G)} \sum_{d|(s,e(g))} \mu(s/d) \#G[d] (1-Y^s)^{n/s}. 
\end{eqnarray}

We single out $N(i,g)$ and obtain:
\begin{eqnarray*}
 (-1)^i N(i,g)= \frac{1}{n} \sum_{s|(\mathrm{exp}(G),i)} \sum_{d|(s,e(g))} \mu(s/d) \#G[d] (-1)^{i/s} {{n/s}\choose{i/s}}.
\end{eqnarray*}
Hence we find
\begin{eqnarray*}
 N(i,g)= \frac{1}{n} \sum_{s|(\mathrm{exp}(G),i)} (-1)^{i+\frac{i}{s}} {{n/s}\choose{i/s}} \sum_{d|(s,e(g))} \mu(s/d) \#G[d].
\end{eqnarray*}
This finishes the proof of Theorem \ref{la}. 

\section{Related problem}

We will now generalize Theorem 1.2 from \cite{WA2}. For $g \in G$ we define $N(i,g)^*=\{S \subseteq G \setminus \{0\}: \#S=i, \sum_{s \in
S}s=g\}$. We have the following result.

\begin{theorem} \label{ars}
For $g \in G$ we have
\begin{eqnarray*}
 N(i,g)^* = \frac{1}{n} \sum_{s|\mathrm{exp}(G)} (-1)^{i+\lfloor\frac{i}{s}\rfloor} {{n/s-1}\choose{\lfloor i/s \rfloor}} \sum_{d|(s,e(g))} \mu(s/d)
\#G[d].
\end{eqnarray*}
\end{theorem}
\begin{proof}
We have the following identity: 
\begin{eqnarray*}
 \sum_{i=0}^n \sum_{g \in G} N(i,g)^* g X^i= \prod_{\sigma \in G, \sigma \neq 0} (1+\sigma X) \in \C[G][X].
\end{eqnarray*}
As all characters have value one on $0$, we deduce the following from Equation \ref{s}:
\begin{eqnarray*}
\sum_{i=0}^n N(i,g)^*(-Y)^i= \frac{1}{n} \sum_{s|\mathrm{exp}(G)} \sum_{d|(s,e(g))} \mu(s/d) \#G[d] (1-Y^s)^{n/s-1}(1+Y+\ldots+Y^{s-1}). 
\end{eqnarray*}
Comparing the coefficients gives the result.
\end{proof}

\section{acknowledgements}
I would like to thank Hendrik Lenstra for helping me to find the proof of the main theorem.

\end{document}